\documentclass{amsart}

\usepackage{etex}
\usepackage{amsthm}
\usepackage{amsmath} 
\usepackage{amssymb}
\usepackage{url}
\usepackage{amsfonts}
\usepackage{cite}
\usepackage{pst-all}
\usepackage{tikz-cd}
\usetikzlibrary{positioning}
\usetikzlibrary{matrix,arrows,decorations.pathmorphing}

\DeclareMathOperator{\coker}{coker}

\DeclareMathOperator{\Ker}{Ker}
\DeclareMathOperator{\Hom}{Hom}

\begin{document}

\title{A Random Bockstein Operator}
\author{Matthew Zabka}
\address{Department of Mathematics, Southwest Minnesota State University, Marshall, MN 56258}

\date{\today}

\newtheorem{thm}{Theorem}[section]
\newtheorem{definition}[thm]{Definition}
\newtheorem{lem}[thm]{Lemma}
\newtheorem{ex}[thm]{Example}
\newtheorem{prop}[thm]{Proposition}
\newtheorem{rmk}[thm]{Remark}

\begin{abstract}
As more of topology's tools become popular in analyzing high-dimensional data sets, the goal of understanding the underlying probabilistic properties of these tools becomes even more important. While much attention has been given to understanding the probabilistic properties of methods that use homological groups in topological data analysis, the probabilistic properties of methods that employ cohomology operations remain unstudied. In this paper, we investigate the Bockstein operator with randomness in a strictly algebraic setting. 
\end{abstract}

\maketitle

\section{Introduction}
Using the tools of algebraic topology to better understand a data set is a relatively new idea with many applications. For example, Carlsson's survey \cite{carlsson2009topology} reviews the generalization of cluster analysis to persistent homology, a technique that provides more information on the shape of a data set than traditional cluster analysis. Other authors, such as Kahle in ~\cite{kahle2011random}, have investigated the topology of a random simplicial complex. 
 
Both of these approaches have only considered Betti numbers, i.e., the ranks of cohomology groups. A natural question that arises is whether one can gather any additional information from a data set by looking at operations on the topological structure generated by that data set.  That is, how can we expand the idea of randomness to cohomolgy operations? The Bockstein homomorphism is a well-known example of a cohomology operator, and in this paper, we shall attempt to investigate this cohomology operator with randomness. 

Cohomology operators are a topological invariants that can reveal additional structure not seen in cohomological groups. For example, the cohomology groups of for $S^1 \vee S^2$ and $\mathbb{RP}^2$ are the same, but they are not homotopy equivalent spaces. To see this, one can compute the Bockstein homomorphism of both $S^1 \vee S^2$, which is trivial, and of $\mathbb{RP}^2$, which is non-trivial.

In general, the Bockstein homomorphism is a connecting homomorphism of cohomology groups defined on a chain complex. Ideally, we should consider the case of a chain complex of a randomly generated topological space. Unfortunately, this problem is very difficult. The length of the chain complex, each Abelian group in the complex, and each boundary map would all add complexity to this model.  We shall therefore examine in this paper a simpler algebraic version of the above problem whose only degrees of freedom are determined by a single boundary map. 

Let $V$ and $W$ be free-modules with coefficients in $\mathbb{Z}/p^2$. We have then have the following short exact sequences
\[
0\to pV \hookrightarrow V \twoheadrightarrow \overline{V} \to 0 
\quad 
\mathrm{and} 
\quad 
0\to pW \hookrightarrow W \twoheadrightarrow \overline{W} \to 0, 
\]
where $\overline{V}$ and $\overline{W}$ are the reductions of $V$ and $W \mod p$.   Given a map $\phi:V\to W$, which is the boundary map we describe in the paragraph above, define $\psi$ from $\overline{V}$ to $\overline{W}$ to be the map induced by $\phi$. The Bockstein homomorphism induced by $\phi$ is then a map from $\ker\psi$ to $\coker\psi$. We give construction of the Bockstein homomorphism for this case in more detail in Section \ref{BocksteinSection}.  

Since Bockstein homomorphisms are elements of $\hom(\ker\psi, \coker\psi)$, it makes sense only to compare Bocksteins induced by functions from $V$ to $W$ that are equal modulo $p$. If $V$ has dimension $n$ and $W$ has dimension $m$, then a choice of random function from $V$ to $W$ is the same as choosing a random $m$ by $n$ matrix.  To this end, let $\phi$ be a random matrix whose entries are chosen i.i.d. randomly from the discrete uniform distribution on $\{0,1,2,\ldots, p^2-1\}$. Let $\psi$ be the reduction of $\phi$ modulo $p$.  Let $\beta_\phi$ the be Bockstein homomorphism induced by $\phi$.   Let $\gamma$ be in $\hom(\ker\psi, \coker\psi)$.  We shall show that

\[
\mathbb{P}\big(\beta_\phi=\gamma \big| \overline{\phi}=\psi\big) = \frac{1}{p^{k(m-n+k)}}.
\] 

In other words, we shall show that, given $\overline{\phi}=\psi$, the Bockstein homomorphisms are distributed uniformly.

\section{Linear Algebra over $\mathbb{Z}/p^2$}\label{Zp2LA}

Many of our calculations will be done over $\mathbb{Z}/p^2$-modules. This section reviews the theory of $\mathbb{Z}/p^2$-modules over $\mathbb{Z}/p^2$.  Some of the techniques used in this section work for modules over rings other than $\mathbb{Z}/p^2$, but we shall not explore these ideas here.

Let $R$ be a ring. Given an $R$-module $M$, we say that a subset $E$ of $M$ is a basis for $M$ whenever $E$ generates $M$ and $E$ is linearly independent. This definition is equivalent to the condition that every $x$ in $M$ can be written as a unique linear combination of elements of $E$ with scalars in $R$.  A module that has a basis is called a free module.

Let $p$ be prime, and let $V$ and $W$ be free $\mathbb{Z}/p^2$-modules. Define
\[ 
\overline{V}:= V \bigotimes_{\mathbb{Z}/p^2} \mathbb{Z}/p 
\quad 
\mathrm{and} 
\quad \overline{W}:= W \bigotimes_{\mathbb{Z}/p^2} \mathbb{Z}/p.
\] 

So $\overline{V} = V/pV$ and $\overline{W} = W/pW$ are the reductions of $V$ and $W \mod p$.  Note that these are $\mathbb{Z}/p$ vector spaces. For an element $x\in V$, we use $\overline{x}$ to denote its reduction modulo $p$. For an element $y$ in $\overline{V}$, we use $\tilde{y}$ to denote a choice of representative in $V$ of $y$, so that $\overline{\tilde{y}}=y$. Given a $\mathbb{Z}/p^2$-linear map $\phi: V\to W$, let $\overline{\phi}$ denote the induced function from $\overline{V}$ to $\overline{W}$.

\begin{lem}\label{KerIsIm}
\normalfont Let $V$ be a free $\mathbb{Z}/p^2$-module. Let $p:V\to V$ be multiplication by $p$. Then the kernel of $p$ is equal to the image of $p$.  
\end{lem}

\begin{proof}
Let $\{e_i\}$ be a basis for $V$. Let $x$ be in $\ker p$. Since $\{e_i\}$ is a basis, there are $\alpha_i$ in $\mathbb{Z}/p^2$ such that $x=\sum_i \alpha_i e_i$. Since $x$ is in $\ker p$ we have $px=\sum_i p\alpha_i\cdot e_i = 0$. By the independence of the $e_i$, we have $p\alpha_i = 0$ for each $i$. Thus $\alpha_i = p\beta_i$ for some $\beta_i\in\mathbb{Z}/p^2$. Then $p(\sum_i\beta_i e_i)=\sum_i \alpha_i e_i = x$. So that $x$ is in the image of $p$.  

Next, assume that $y$ is in the image of $p$.  Then there exists a $z\in V$ with $pz=y$. So $py = p^2 z = 0$. So $y$ is in the kernel of $p$.
\end{proof}

We know that $pV$ and $\overline{V}$ are isomorphic as $\mathbb{Z}/p$-vector spaces, because they both have the same dimension. The following lemma gives an explicit isomorphism between these two spaces.

\begin{lem}\label{pxToBarx}
\normalfont The map $f:pV \to \overline{V}$ defined by $px \mapsto \overline{x}$ is a $\mathbb{Z}/p$-linear isomorphism.
\end{lem}

\begin{proof}
We show that both $f$ and its inverse mapping $g$, which maps $\overline{x}$ in $\overline{V}$ to $px$ in $pV$, are well-defined. To show that $f$ is well-defined, assume that $px = py$ for some $x$ and $y$ in $V$. Then $px-py= p(x-y) = 0$. So $x-y = pz$ for some $z\in V$  by Lemma \ref{KerIsIm}. Note that
\[
\overline{x}-\overline{y} = \overline{x-y} = p\overline{z} = \overline{0},
\] 
so that $f$ is well-defined.

For the inverse mapping $g$, suppose $\overline{x} = \overline{y}$. Then $\overline{x-y} = \overline{0}$. So by Lemma \ref{KerIsIm}, $x-y = pz$ for some $z\in V$.  We have 
\[
px - py = p(x-y) = p^2 z = 0. 
\]
So $g$ is well-defined. By inspection we see that both $f$ and $g$ are $\mathbb{Z}/p$-linear functions, and so the proof is complete.
\end{proof}

The main proposition of this section shows that any lift of a basis of $\overline{V}$ is a basis of $V$. Such bases will be useful for constructing linear maps out of $V$. That is, if one defines a map on any basis of $V$, then this map extends linearly to all of $V$.

\begin{prop}\label{LiftProp}
\normalfont Let $\{e_i\}$ be a basis for $\overline{V}$. For each $e_i$, let $\tilde{e}_i$ in $V$  be any lift of $e_i$.  Then $\{\tilde{e}_i\}$ is a basis for $V$.
\end{prop}

\begin{proof}
We first show that the set $\{\tilde{e}_i\}$ is linearly independent.  Suppose $\alpha_i\in \mathbb{Z}/p^2$ with 
\begin{equation}\label{LiftEq}
\sum_i \alpha_i \tilde{e}_i = 0.
\end{equation}
Projecting to $\overline{V}$ we obtain $\sum_i \overline{\alpha}_i e_i = \overline{0}$. Since $\{e_i\}$ is a basis for $\overline{V}$, we must have that $\overline{\alpha_i} = \overline{0}$ for every $i$. So each $\alpha_i = p\beta_i$ for some $\beta_i$ in $\mathbb{Z}/p^2$. Thus, (\ref{LiftEq}) gives that $\sum_i \beta_i \cdot p\tilde{e}_i = 0$ in $pV$. Under the isomorphism given in Lemma \ref{pxToBarx}, we have $\sum_i \overline{\beta_i}e_i = \overline{0}$ in $\overline{V}$. Since the set $\{e_i\}$ is linearly independent, each $\overline{\beta_i} = \overline{0}$, so each $\beta_i = p\gamma_i$ for some $\gamma_i$ in $\mathbb{Z}/p^2$. This gives that each $\alpha_i = p\beta_i =  p^2\gamma_i=0$. So the set $\{\tilde{e}_i\}$ is linearly independent.

We next show that $\{\tilde{e}_i\}$ spans $V$. Let $x\in V$. Since the set $\{e_i\}$ is a basis for $\overline{V}$, there are $\alpha_i\in\mathbb{Z}/p^2$ such that $\sum_i\overline{\alpha_i} e_i = \overline{x}.$ So for some $y\in V$,

\begin{equation}
x  =  py + \sum_i \alpha_i \tilde{e}_i.
\end{equation}

Under the isomorphism given in Lemma \ref{pxToBarx}, the element $py$ in $pV$ is mapped to $\overline{y}$ in $\overline{V}$. Since the $e_i$ form a basis for $\overline{V}$, there exist $\beta_i$ in $\mathbb{Z}/p^2$ such that $\sum_i \overline{\beta_i} e_i = \overline{y}$. Thus $pz + \sum_i \beta_i \tilde{e}_i = y$ for some $z\in V$. Substituting this into (2) gives

\[
x = p\left(pz + \sum\beta_i\tilde{e}_i  \right) +\sum_i \alpha_i \tilde{e}_i.
\]
Simplifying gives $x = \sum_i (\alpha_i - p\beta_i)\tilde{e}_i$, so that $x$ is in the span of $\{\tilde{e}_i\}$, as desired.
\end{proof}

For the map $\psi$ with domain $\overline{V}$ and target $\overline{W}$, recall that $\coker\psi$ is defined as the quotient $\overline{W}/\psi(\overline{V})$.  Our next lemma shows that we may regard the Bockstein homomorphism as a map $\beta:\ker \psi \to \coker \psi$. The techniques used in the proof are similar to the techniques used in Lemma \ref{pxToBarx}.

\begin{lem}\label{CoKerIso}
\normalfont The map $f$ from $pW/\phi(pV)$ to $\coker \psi$ defined by 
\[
f: pw +\phi(pV) \to \overline{w} + \psi(\overline{V})
\]
is an isomorphism.
\end{lem}

\begin{proof}

We must show that $f$ and its inverse mapping $g$ are well-defined.

To show that $f$ is well-defined, suppose $pw+\phi(pV)=pw'+\phi(pV)$ in $pW/\phi(pV)$. We must show that $\overline{w}-\overline{w'}$ is in $\psi(\overline{V})$. We have that $p(w-w') \in \phi(pV)$.  Thus $p(w-w') = p\phi(v)$ for some $v\in V$. By Lemma \ref{KerIsIm}, we have $w-w'-\phi(v) = py$. Thus $\overline{w} - \overline{w'} = \overline{\phi(v)}=\psi(\overline{v})$. So $\overline{w} - \overline{w'}$ is in $\psi(\overline{V})$, and this shows that $f$ is well-defined.

We next want to show that the inverse mapping $g$ is well-defined.  Suppose that $\overline{w}+\psi(\overline{V}) = \overline{w'}+\psi(\overline{V})$.  We must show that $pw+\phi(pV) = pw' +\phi(pV)$.Since $\overline{w} - \overline{w'} + \overline{\phi}(\overline{V}) = \overline{0}+\overline{\phi}(\overline{V})$, there exists a $\overline{v}\in\overline{V}$ with $\overline{w}-\overline{w'}=\overline{\phi}(\overline{v})$.  Thus $w-w'-\phi(v) = px$ for some $x$, which, by Lemma \ref{KerIsIm} gives $p[w-w'-p\phi(v)] = 0$. So $pw+\phi(pV) = pw' +\phi(pV).$ By inspection, $f$ and $g$ are both linear, and the proof is complete.
\end{proof}

\section{Spaces of Linear Maps}\label{L_0Section}

We should like to further investigate the connection between a map $\psi:\overline{V} \to \overline{W}$ and the Bockstein homomorphisms induced by a map $\phi:V\to W$ such that $\overline{\phi}=\psi$. For this section, we shall treat $\psi$ as a fixed $\mathbb{Z}/p$-linear map from $\overline{V}$ to $\overline{W}$.

\begin{definition}\label{LpsiDef}
\normalfont Let $V$ and $W$ be $\mathbb{Z}/p^2$-modules. Let $\overline{V}$ and $\overline {W}$ be the reductions of $V$ and $W$ modulo $p$. Let $\psi$ be a fixed $\mathbb{Z}/p$-linear map from $\overline{V}$ to $\overline{W}$. Define $L_\psi$ to be the collection of all maps from $V$ to $W$ whose reduction modulo $p$ is $\psi$.
\end{definition}

It will also be useful in this section to choose a basis for $\overline{V}$, which, by Proposition \ref{LiftProp} will lift to a basis for $V$.

\begin{definition}\label{BasisDef}
\normalfont Let $V$, $\overline{V}$, and $\psi$ be as in Definition \ref{LpsiDef}. Let $\{e_i\}\cup\{f_j\}$ be a basis for $\overline{V}$ such that $\{e_i\}$ is a basis for the subspace $\ker\psi$ of $\overline{V}$. For each $i$, let $\tilde{e_i}$ in $V$ be a lift of $e_i$. For each $j$ let $\tilde{f_j}$ in $V$ be a lift of $f_j$. 
\end{definition} 

By Proposition \ref{LiftProp}, $\{\tilde{e_i}\}\cup\{\tilde{f_j}\}$ is a basis for $V$.  If the map $\psi:\overline{V}\to\overline{W}$ is not the zero map, then we know that $L_\psi$ is not a vector space, for in this case, $0$ is not in $L_\psi$.  This fact, along with the next lemma, gives that $L_\psi$ is a vector space if and only if $\psi$ is the zero map.

\begin{lem}\label{L_0isZp}
\normalfont
The space $L_0$ with pointwise addition and $\mathbb{Z}/p$ scalar multiplication defined by
\[
\overline{\alpha}\cdot \phi := \alpha \cdot \phi,
\]
where $\alpha$ is in $\mathbb{Z}/p^2$ and $\phi$ is in $L_0$, is a $\mathbb{Z}/p$-vector space.  In particular, if $V$ has dimension $n$ and $W$ has dimension $m$, then $L_0$ is a $\mathbb{Z}/p$-vector space of dimension $m\cdot n$.
\end{lem}

\begin{proof}
We shall only show that this scalar multiplication is well-defined, as the other parts of the proof are straightforward.  Let $\alpha_1$ and $\alpha_2$ be in $\mathbb{Z}/p^2$ with $\overline{\alpha_1}=\overline{\alpha_2}$. Let $\phi$ be in $L_0$ and let $v\in V$.  Then $\alpha_1-\alpha_2 = p\beta$ for some $\beta$ in $\mathbb{Z}/p^2$ and $\phi(v) = pw$ for some $w$ in $W$, because $\overline{\phi} = 0$. So we have
\begin{eqnarray*}
\overline{\alpha_1}\cdot\phi(v) - \overline{\alpha_2}\cdot\phi(v) &=& \alpha_1\phi(v) - \alpha_2\phi(v) \\
   &=& (\alpha_1-\alpha_2)\phi(v) \\
   &=& (p\beta) (pw)\\
   &=& p^2 \beta w\\
   &=& 0,
\end{eqnarray*}
which shows that this scalar multiplication in $\mathbb{Z}/p^2$ is well-defined.
\end{proof}

Let $\phi_0$ be any element of $L_\psi$. Then $\phi_0 + L_0 = L_\psi$, so we may regard $L_\psi$ as a coset of $L_0$. It will be useful however to choose a particular $\phi_0\in L_\psi$ whenever we wish to regard $L_\psi$ as a coset of $L_0$. For this, we need only define $\phi_0$ on the basis $\{\tilde{e_i}\}\cup\{\tilde{f_j}\}$ given in Definition \ref{BasisDef}.

\begin{rmk}\label{Phi0Def} \normalfont
When we regard $L_\psi$ as a coset of $L_0$, we shall choose $\phi_0$ such that, for all $i$, $\phi_0(\tilde{e}_i) = 0$, and for all $j$, $\phi_0(\tilde{f_j})$ is any value whose reduction modulo $p$ is $\psi(f_j)$. 
\end{rmk}

We are now ready to count the number of elements in $L_\psi$.

\begin{lem}\label{SizeLpsi}
\normalfont
For any $\psi:\overline{V}\to\overline{W}$, the set $L_\psi$ has $p^{m n}$ elements.
\end{lem}

\begin{proof}
Lemma \ref{L_0isZp} tells us that $L_0$ is a $\mathbb{Z}/p$-vector space, but by definition, $L_0$ also is a $\mathbb{Z}/p^2$-submodule of $\Hom(V,W)$.  When we regard $L_\psi$ as $\phi_0+L_\psi$, where $\phi_0$ is as defined in Remark \ref{Phi0Def}, this addition occurs in a $\mathbb{Z}/p^2$-submodule. So, while $L_\psi$ is not a translate of $L_0$ as a $\mathbb{Z}/p$-vector space, we still know that $L_\psi$ has the same number of elements as $L_0$. This information, along with Lemma \ref{L_0isZp}, completes the proof.
\end{proof}


\section{The Bockstein Homomorphism}\label{BocksteinSection}

What follows is a short review of the Bockstein homomorphism in the context that is relevant for our study of cohomology operations with randomness. Several references cover the Bockstein homomorphism and cohomology operations in more generality. See, for example, \cite{mosher2008cohomology}.

As in Section \ref{Zp2LA}, let $V$ and $W$ be $\mathbb{Z}/p^2$ free-modules with coefficients in $\mathbb{Z}/p^2$. We have the following short exact sequences:
\[
0\to pV \hookrightarrow V \twoheadrightarrow \overline{V} \to 0 
\quad 
\mathrm{and} 
\quad 
0\to pW \hookrightarrow W \twoheadrightarrow \overline{W} \to 0, 
\]
where $\overline{V}$ and $\overline{W}$ are the reductions of $V$ and $W \mod p$. 

Consider a $\mathbb{Z}/p^2$-linear map $\phi$ from $V$ to $W$. Let $\psi$ be the map from $\overline{V}$ to $\overline{W}$ induced by $\phi$. Then the Snake Lemma \cite{atiyah1994introduction} defines a map $\beta$ with domain $\ker\psi$ and target $pW/\phi(pV)$. The following diagram illustrates the Snake Lemma.
\begin{center}
\begin{tikzpicture}[>=angle 90,scale=2.2,text height=1.5ex, text depth=0.25ex]

\node (k1)  at (0,3) {} ;
\node (k2) [right=of k1] {};
\node (k3) [right=of k2] {$\Ker \psi$};
\node (a1) [below=of k1] {$pV$};
\node (a2) [below=of k2] {$V$};
\node (a3) [below=of k3] {$\overline{V}$};
\node (b1) [below=of a1] {$pW$};
\node (b2) [below=of a2] {$W$};
\node (b3) [below=of a3] {$\overline{W}$};
\node (c1) [below=of b1] {$pW/\phi(pV)$};
\draw[->,red]
(k3) edge[out=0,in=180,red] node[pos=0.55,yshift=5pt] {$\beta$} (c1);
\draw[->]
(k3) edge (a3)
(b1) edge (c1);
\draw[->]
(a1) edge node[auto] {} (a2)
(a2) edge node[auto] {} (a3)
(a1) edge node[auto] {} (b1)
(a2) edge node[auto] {$\phi$} (b2)
(a3) edge node[auto] {$\psi$} (b3)
(b1) edge node[below] {} (b2)
(b2) edge node[below] {} (b3);
\end{tikzpicture}
\end{center}

More precisely, for $\overline{v}\in\ker\psi$, choose any representative $v\in V$ of $\overline{v}$. Since the squares in the above diagram commute, we have $\overline{\phi(v)} = \psi (\overline{v}) = \overline{0}$. So $\phi(v) = pw$ for some $w\in W$. Define the Bockstein homomorphism $\beta$ from $\ker\psi$ to $pW/\phi(pV)$ by 
\[
\beta(\overline{v}): = pw + \phi(pV).
\]
The following diagram shows the process described above.

\[
\begin{tikzcd}
\			& v\arrow[mapsto]{r}\arrow[mapsto]{d}{\phi} &\overline{v}\\
pw\arrow{d}\arrow[mapsto]{r} & pw = \phi(v) \\
pw + \phi(pV)
\end{tikzcd}
\]

By construction, the target of $\beta$ is $pW/\phi(pV)$. However, by Lemma \ref{CoKerIso}, we know that $pW/\phi(pV)$ is isomorphic to $\coker \psi$. So henceforth we shall regard $\beta$ as a map into $\coker\psi$.

\begin{rmk}\normalfont
We note here that if one regards an arbitrary chain complex, the map $\beta$ is often called a connecting homomorphism. When the chain complex is generated by a topological space, the map $\beta$ is called the Bockstein homomorphism.  If we regard $\phi$ as the map between $V$ and $W$ in the following chain complex

\[
\begin{tikzcd}
\cdots \arrow{r} & 0 \arrow{r}&V\arrow{r}{\phi}& W \arrow{r} & 0 \arrow{r}&\cdots,
\end{tikzcd}
\]
and consider the reduced chain complex
\[
\begin{tikzcd}
\cdots \arrow{r} & 0 \arrow{r}&\overline{V}\arrow{r}{\psi}& \overline{W} \arrow{r} & 0 \arrow{r}&\cdots,
\end{tikzcd}
\]
then the only possible non-trivial homology groups of this chain complex are $\ker \psi$ and $\coker\psi$.  Although we are in a strictly algebraic setting, we shall continue to refer to the map $\beta$ as the Bockstein homomorphism between $\ker \psi$ and $\coker \psi$.
\end{rmk}

\begin{rmk} \normalfont
The Bockstein homomorphism is often constructed in the case where $V$ and $W$ are $\mathbb{Z}$-modules. In this case, first reduce $V$ and $W$ to $\mathbb{Z}/p^2$ modules, and then apply the above construction.
\end{rmk}

In this section we have described how to every $\phi\in L_\psi$, there is a unique Bockstein homomorphism $\beta_\phi:\ker\psi \to \coker \psi$. This fact defines the following map.
\begin{definition}\label{DefBetaPhi}
\normalfont
Define $\Gamma$ to be the map from $L_\psi$ to $\Hom(\ker\psi,\coker\psi)$ that sends $\phi$ in $L_\psi$ to the unique Bockstein homomorphism $\beta_\phi$, which is in $\Hom(\ker\psi,\coker\psi)$, that is given by $\phi$.
\end{definition}

Composing $\Gamma$ with addition by $\phi_0$ gives a well defined set map $B$ with domain $L_0$ and target $\Hom(\ker\psi, \coker\psi)$.
This is shown in the following diagram.
\[
\begin{tikzcd}
L_0 \arrow{r}{+\phi_0} \arrow[bend left, red]{rr}{B} & L_\psi \arrow{r}{\Gamma} & \Hom(\ker\psi, \coker\psi)
\end{tikzcd}
\]

We should like to examine the properties of this map. The map from $L_0$ to $L_\psi$ given by adding $\phi_0$ is a bijection. The next lemma shows that the map $\Gamma$ is onto, which shows that $B$ is also onto.  In particular, every $\mathbb{Z}/p$ linear map from the kernel of $\psi$ to the cokernel of $\psi$ is the Bockstein homomorphism of some $\phi:V\to W$ that induces $\psi$.

\begin{lem}\label{Bonto}\normalfont
The map $\Gamma$ from $L_\psi$ to $\Hom(\ker\psi, \coker\psi)$ is onto.
\end{lem}

\begin{proof}
Let $\beta\in\Hom(\ker\psi,\coker\psi)$. Let $\{e_i\}\cup\{f_j\}$ and $\{\tilde{e_i}\}\cup\{\tilde{f_j}\}$ be bases of $\overline{V}$ and $V$ as defined in Definition \ref{BasisDef}. We shall define $\phi$ on the basis for $V$ and then extend linearly to define $\phi$ on all of $V$. We must show that the Bockstein homomorphism $\beta_\phi$ of $\phi$ is equal to $\beta$.

For each $i$, we know that $e_i$ is in the domain of $\beta$. So $\beta(e_i) = \overline{w_i} + \psi(\overline{V})$ for some $w_i$ in $W$. Define $\phi(\tilde{e_i}) = pw_i$.  Define $\phi(\tilde{f_j})$ to be any value in $W$ whose reduction modulo $p$ is $\psi(f_j)$.  Then $\overline{\phi(e_i)} = \overline{pw_i} = 0 = \psi(e_i)$ and $\overline{\phi(\tilde{f_j})} = \psi(f_j)$. This shows that $\overline{\phi}=\psi$. In particular $\phi$ is in $L_\psi$.

By construction, $\beta_\phi(e_i) = \overline{w_i}+\psi(\overline{V}) = \beta(e_i)$. Since $\beta_\phi$ is equal to $\beta$ on the basis of $\ker\psi$, they are equal as $\mathbb{Z}/p$-linear functions. 
\end{proof}

\begin{lem}
\normalfont The map $B$ is a $\mathbb{Z}/p$-linear map.  
\end{lem}

\begin{proof}
Let $\phi$ and $\phi'$ be in $L_0$. We must show that $\Gamma(\phi+\phi'+\phi_0) = \Gamma(\phi + \phi_0) + \Gamma(\phi'+\phi_0)$, for $\phi_0\in L_\psi$ as described in Remark \ref{Phi0Def}. For a basis $\{e_i\}$ of $\ker\psi$, it suffices to show that
\[
\Gamma(\phi+\phi'+\phi_0)(e_i) = \Gamma(\phi + \phi_0)(e_i) + \Gamma(\phi'+\pi_0)(e_i)
\]

Let $\tilde{e_i}$ be any lift of $e_i$.  Then $\phi_0(\tilde{e_i})=0$ by construction. Also,  
\begin{equation}\label{Gamma1}
\overline{(\phi+\phi')(\tilde{e_i})}= \overline{\phi(\tilde{e_i})} + \overline{\phi'(\tilde{e_i})} =0,
\end{equation}
and
\begin{equation}\label{Gamma2}
\overline{\phi(\tilde{e_i})}=\overline{\phi'(\tilde{e_i})}=0,
\end{equation}
because $\phi$ and $\phi'$ are in $L_0$. Thus there are $w_i$ and $w_i'$ in $W$ with $\phi(\tilde{e_i}) = pw$ and $\phi'(\tilde{e_i}) = pw'$. Thus by (\ref{Gamma2}) we know that 

\begin{eqnarray*}
\Gamma(\phi + \phi_0)(e_i) + \Gamma(\phi'+\phi_0)(e_i) &=&\left( \overline{w_i} + \psi(\overline{V})\right) + \left(\overline{w'_i} + \psi(\overline{V})\right)\\
 &=& \overline{w_i} + \overline{w'_i} + \psi(\overline{V}).
\end{eqnarray*}
Equations (\ref{Gamma1}) and (\ref{Gamma2}) together give that
\[
\Gamma(\phi+\phi'+\phi_0)(e_i) = \overline{w_i} + \overline{w_i'} + \psi(\overline{V}).
\]
So $\Gamma(\phi+\phi'+\phi_0)(e_i) = \Gamma(\phi + \phi_0)(e_i) + \Gamma(\phi'+\phi_0)(e_i)$, as desired.
\end{proof}

\section{Counting}

Let $V$ and $W$ be $\mathbb{Z}/p^2$-modules of dimensions $n$ and $m$ respectively. In Section \ref{L_0Section} we defined $L_\psi$ as the collection of all maps from $V$ to $W$ whose reduction modulo $p$ is $\psi$. We then found that $L_\psi$ has $p^{m\cdot n}$ elements. We should next like to answer the following question: Given a Bockstein homomorphism $\beta$, which is in $\hom(\ker\psi \coker\psi)$, how many $\phi$ in $L_\psi$ have $\beta$ as their Bockstein homomorphisms?  To answer this question, we shall first look at the size of $\Gamma^{-1}(\beta)$.  

\begin{lem}\label{SizeGammaInv}
\normalfont
Let $k:=\dim(\ker\psi)$. Then the space $\Gamma^{-1}(\beta)$ has exactly $p^{(m+k)(n-k)}$ elements.
\end{lem}

\begin{proof}
Since translation by the $\phi_0$ given in Remark \ref{Phi0Def} is a bijection, we know that $B^{-1}(\beta)$ has the same size as $\Gamma^{-1}(\beta)$. Since $B$ is a linear map, we also know that $B^{-1}(0)$ has the same size as $B^{-1}(\beta)$.  Thus $\Gamma^{-1}(\beta)$ has the same size as $B^{-1}(0)$, so we shall find the size of $B^{-1}(0)$.

Since $V$ and $W$ have dimension $n$ and $m$ respectively, we know by Proposition \ref{LiftProp} that $\overline{V}$ and $\overline{W}$ also have dimensions $n$ and $m$ respectively.   Recall that $\coker\psi$ is defined as $\overline{W}/ \psi(\overline{V})$. Let $k:=\ker(\psi)$. Since $\psi$ is a $\mathbb{Z}/p$-linear map, by the Rank-Nullity Theorem, we have that $n = k + \dim(\psi(\overline{V}))$. So $\dim(\psi(\overline{V})) = n-k$.  Thus $\dim(\coker(\psi)) = m-(n-k)$.  From this, we have that the number of elements in $B^{-1}(0)$ is 
\[
p^{mn-k(m-n+k)}=p^{(m+k)(n-k)}.
\]
\end{proof}

We now come to our main result. Recall that a choice of random function from $V$ to $W$ is the same as choosing a random $m$ by $n$ matrix.  

\begin{thm}
\normalfont
Let $\phi$ be an $m$ by $n$ matrix whose entries are chosen i.i.d. from the discrete uniform distribution on $\{0,1,2,\ldots, p^2-1\}$. Let $\psi=\overline{\phi}$.  Let $\beta_\phi$ be the Bockstein homomorphism defined by $\phi$ as in Definition \ref{DefBetaPhi}.  Note that $\beta_\phi$ is a random variable. Let $\beta$ be in $\hom(\ker\psi,\coker\psi)$. Then
\[
\mathbb{P}\big(\beta_\phi=\beta \big| \overline{\phi}=\psi\big) = \frac{1}{p^{k(m-n+k)}}.
\]
\end{thm}

\begin{proof}
We know from Remark \ref{SizeLpsi} that $L_\psi$ has $p^{mn}$ elements.  By Lemma \ref{Bonto}, we know that $\Gamma$ is onto, and by Lemma \ref{SizeGammaInv}, we know that the size of $\hom(\ker\psi,\coker\psi)$ is $p^{mn-k(m-n+k)}$.
\end{proof}

\bibliography{master}
\bibliographystyle{plain}

\end{document}